\documentclass{amsart}       
\usepackage{graphicx}
\usepackage{amsmath}
\usepackage{amsfonts}
\usepackage{fullpage}

\def\d{\partial}
\def\M{\mathcal M}
\def\PP{\mathbb{P}}

\DeclareMathSymbol{\Delta}{\mathalpha}{operators}{1}

\newtheorem{thm}{Theorem}

\newtheorem{prop}[thm]{Proposition}

\newtheorem{defin}[thm]{Definition}

\begin{document}

\title{Simplicial equations for the moduli space\\ of stable rational curves}

\author[J. Maya and J. Mostovoy]{Joaquin Maya         \and
        Jacob Mostovoy
}




\maketitle

\begin{abstract}
In this, largely expository, note, we show how the simplicial structure of the moduli spaces of stable rational curves with marked points allows to produce explicit equations for these spaces. The key argument is an elementary combinatorial statement about the sets of trees with marked leaves.

\end{abstract}

\section{Introduction: $\Delta$-sets}
\label{intro}

There exists a convenient combinatorial notion which allows to encode the structure of a triangulated topological space; namely, that of a \emph{$\Delta$-set} (Rourke, Sanderson 1971).  A $\Delta$-set $X$ is a sequence of sets $X_0, X_1, X_2, \ldots$ together with the maps 
$$
\d_{i}:X_{{n}}\rightarrow X_{n-1},
$$
which are defined for all $n>0$ and $i = 0,1,..., n$, and satisfy
\begin{equation}\label{simplicial}
\d_{i}\circ \d_{j} = \d_{j-1}\circ \d_{i}
\end{equation}
whenever $i < j$. This definition is a simplification of the standard definition of a \emph{simplicial set}, a fundamental notion in algebraic topology and homological algebra, see, for instance, (May 1967) or (Weibel, 1994)


Given a simplicial complex $\mathcal{K}$ with a totally ordered set of vertices, let $X_n$ be the set of all $n$-dimensional simplices of $\mathcal{K}$. For $x\in X_n$ define $\d_i(x)$ to be the $(n-1)$-dimensional face of the simplex $x$ which does not contain the $i$th vertex of $x$. The identities (\ref{simplicial}) are then satisfied and $\mathcal{K}$ gives rise to a $\Delta$-set $X$. Not all $\Delta$-sets come from simplicial complexes; the simplest example is the $\Delta$-set $O$ such that $O_0$ and $O_1$ are one-point sets and $O_n$ is empty for $n>1$.

\begin{defin}We say that a $\Delta$-set $X$ is {uniquely fillable in dimension $n$} if for each sequence
$(x_0, x_1,\ldots, x_n)$ of elements of $X_{n-1}$ that satisfies
$$\d_{i}(x_j) = \d_{j-1}(x_{i})$$
for all $0\leq i < j\leq n$, there exists a unique element $y\in X_n$ with $\d_i(y)=x_i$. 
\end{defin}

If $X$ is uniquely fillable in dimension $n$, the set $X_n$ can be given by the system of equations $\d_{i}(x_j) = \d_{j-1}(x_{i})$ inside the product of $n+1$ copies of $X_{n-1}$. In this note we shall see that this observation can be used in order to produce the equations for various algebraic varieties such as the Deligne-Mumford compactification $\overline{\M}_{0,n}$ of the moduli space of rational curves with $n$ marked points. The main argument is, actually, a combinatorial statement about certain sets of trees.

\section{The $\Delta$-set of trees with marked leaves}
\label{sec:1}

For $n\geq 0$, let $T_{n}$ be the set of all trees without bivalent vertices whose leaves are labelled by the numbers from $0$ to $n$. In particular, $T_0$, $T_1$ and $T_2$ are one-point sets, $T_3$ has 4 elements and $T_4$ consists of 26 elements:

\smallskip

$$\includegraphics[width=200pt]{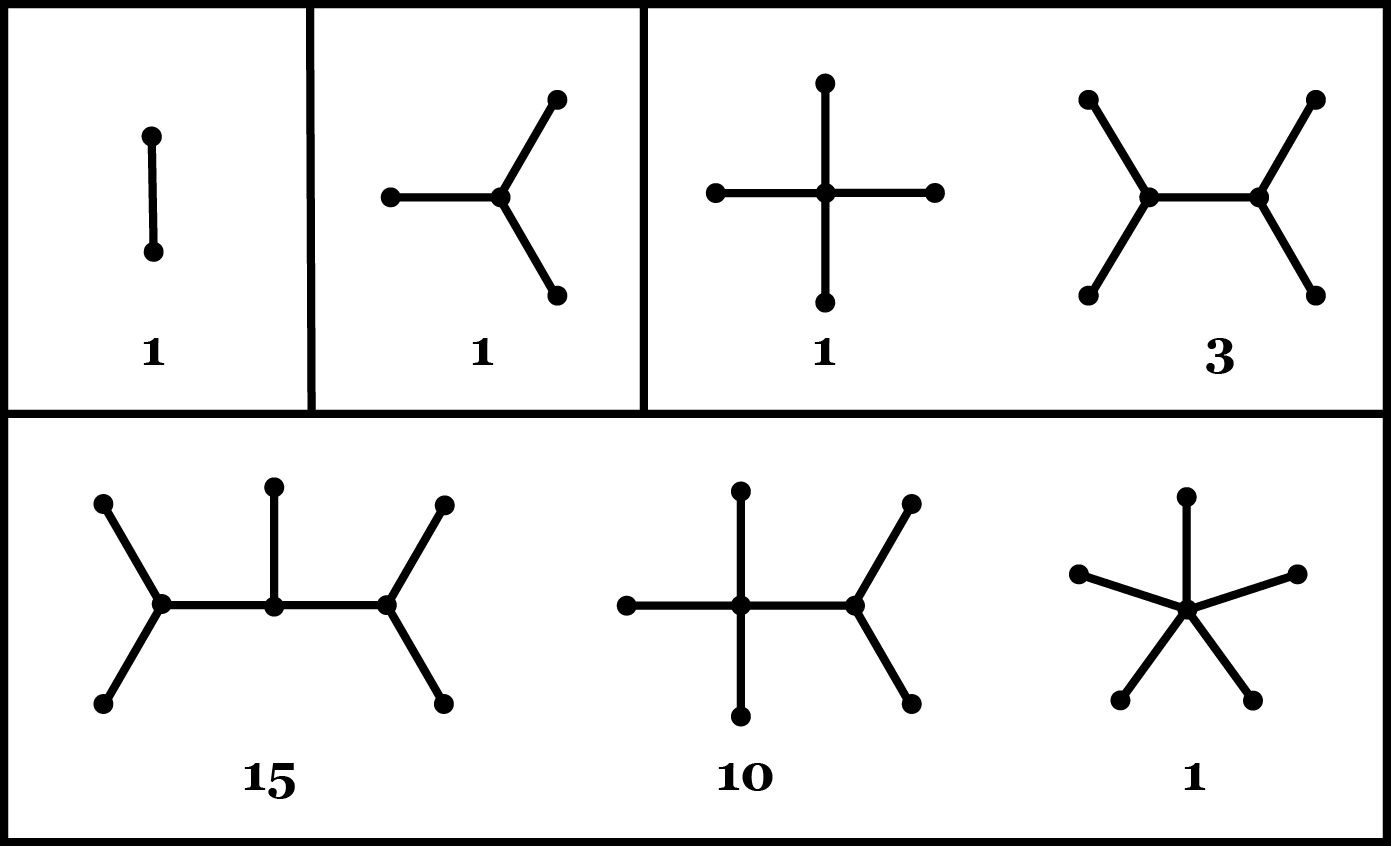}$$ 

For each $i$ between $0$ to $n$ define
$$\d_i: T_n\to T_{n-1}$$
as the map that erases the $i$th leaf and, for $j>i$, replaces the label $j$ by $j-1$. If the resulting tree has a bivalent vertex, it is simply ``smoothed out'': the vertex is deleted and the incoming edges are joined together.

\begin{thm}\label{combin}
The sets $T_n$ together with the maps $\d_i$ form a $\Delta$-set which is uniquely fillable in dimensions 5 and greater.
\end{thm}

The fact that the $T_n$ form a $\Delta$-set is clear. For the purposes of the argument which establishes  the unique fillability in dimension $n$, it will be more convenient to label the leaves of a tree with some fixed labels, rather than number them from 0 to $n$. Namely, consider a set $A$ with $n+1$ elements. We will assume that the leaves of the trees in $T_n$ are marked by distinct elements of $A$; for $\mu\in A$ we write $\d\mu$ for the operation of deleting the leaf labelled by $\mu$ followed, if necessary, by smoothing a bivalent vertex.

\medskip

Consider a tree $t\in T_n$. We will be interested in the following question: for which pairs of labels $\alpha,\beta$  can the tree $t$ be uniquely reconstructed from $\partial_\alpha t$ and $\partial_\beta t$? The answer is expressed in terms of the \emph{adjacency} of leaves in a tree. 
 
Denote by $v(\alpha)$ the vertex to which the leaf $\alpha$ of $t$ is connected. We shall call the leaves $\alpha$ and $\beta $ of the tree $t$ {adjacent} if either $v(\alpha)=v(\beta)$ or $v(\alpha)$ and $v(\beta)$ are both trivalent and connected by an edge. 
Then, the tree $t$ can be uniquely reconstructed from $\partial_\alpha t$ and $\partial_\beta t$ if and only if the leaves $\alpha$ and $\beta$ of $t$ are not adjacent. Indeed, the following three configurations of the adjacent leaves $\alpha$ and $\beta$ cannot be distinguished after one of these leaves is erased:
\smallskip

$$\includegraphics[width=150pt]{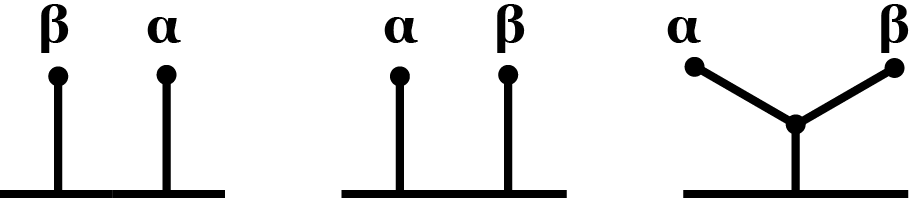}$$ 
\ 
\medskip

\noindent On the other hand, assume that $\alpha$ and $\beta$ are not adjacent in $t$ and consider the tree $\partial_\alpha\partial_\beta t$. The trees $\partial_\alpha t$ and $\partial_\beta t$ are obtained from it by adding one leaf. Each of these leaves is attached either at an internal vertex of $\partial_\alpha\partial_\beta t$ (that is, a vertex of valency greater than 1) or in the interior of an edge, say, at the midpoint. They cannot be attached at the same point since in this case the leaves $\alpha$ and $\beta$ would be adjacent in $t$; this means that both of them can be added simultaneously and the result coincides with $t$.

Now, let us proceed to the proof of the Theorem.  For $n>4$, consider a collection of $n+1$ trees $(x_\mu), \mu\in A$,  each with $n$ marked leaves, such that the leaves of $x_\mu$ have labels in $A-\{\mu\}$.  Assume that
\begin{equation}\label{sim}\d_\alpha x_\beta = \d_\beta x_\alpha\end{equation}
for all pairs of distinct $\alpha, \beta \in A$. We must prove that the exists a unique tree $y$ whose leaves are labelled by the elements of $A$, such that $x_\mu=\d_\mu y$.

Assume that for $\alpha, \beta\in A$ the leaves $\alpha$ and $\beta$ are not adjacent in the tree $x_\mu$ for some $\mu\neq\alpha,\beta$. Then, take $z=\partial_{\alpha} x_{\beta} = \partial_{\beta} x_{\alpha}$. In order to obtain $x_{\alpha}$ and $x_{\beta}$ from $z$ one has to attach the leaves $\alpha$ and $\beta$, respectively, to $z$ at two different points; hence, both of them can be added simultaneously so as to obtain an element $y\in T_n$ with $\partial_{\alpha} y = x_{\alpha}$ and $\partial_{\beta} y = x_{\beta}$. 
We have  
$$\partial_{\alpha} x_\mu = \partial_{\mu} x_{\alpha} = \partial_{\mu}\partial_{\alpha} y = \partial_{\alpha} \partial_\mu y,$$
and similarly, that $\partial_{\beta} x_\mu = \partial_{\beta} \partial_\mu y$. Since the leaves $\alpha$ and $\beta$ are not adjacent in $x_\mu$, this implies that $\partial_\mu y = x_\mu$. 

If the leaves $\alpha$ and $\beta$ are not adjacent in $x_\mu$ for \emph{each} $\mu\neq\alpha,\beta$, the existence (and uniqueness) of $y$ is established. We shall now see that for a ``generic'' solution $(x_\mu)$ of the equations (\ref{sim}) one can find such a pair of non-adjacent labels, and that in the remaining cases the graphs involved are particularly simple, and the existence of $y$ can be established directly.

We can distinguish several cases.

If for each $\mu\in A$ the graph $x_{\mu}$ has only one internal vertex, then  $x_\mu= \partial_\mu y$ where $y$ also has only one internal vertex and the leaves of $y$ are labelled by elements of $A$.

Let the maximal number of the internal vertices of the $x_{\mu}$ be two. Then, each of those $x_\mu$ that has two internal vertices, gives a decomposition of $A-\{\mu\}$ into two disjoint subsets; namely, the sets of leaves attached to each of the internal vertices. 

Assume that the labels $\alpha$ and $\beta$ belong to the same subset with respect to this decomposition of $A-\{\mu\}$ for some $\mu$. Then, it follows from the condition (\ref{sim}) that this is true for \emph{any} label $\mu \neq \alpha,\beta$ such that $x_{\mu}$ has two internal vertices. 
As a consequence, there is a well-defined decomposition of $A$ into two subsets. If $y$ is a graph with two internal vertices that corresponds to this decomposition of $A$, then we have $x_\mu= \partial_\mu y$ for all $\mu \in A$.

Now, assume that  the maximal number of internal vertices of the $x_{\mu}$ is three and $n=5$. One verifies directly that all the solutions are of the type $(\partial_\mu y)$ where $y$ is one of the following graphs: 
\smallskip

$$\includegraphics[width=400pt]{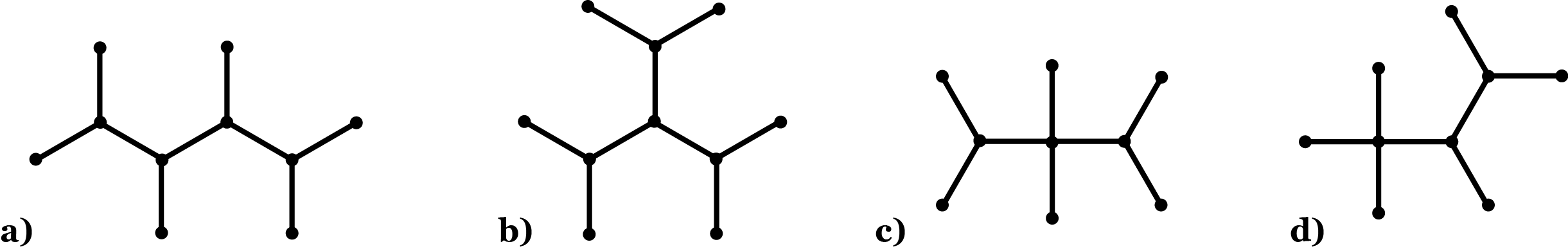}$$ 
\ 
\medskip

Finally, consider the case when at least one of the $x_{\mu}$ has more than two internal vertices and $n>5$.
In this situation, we can always find two labels $\alpha$ and $\beta$ such that the corresponding leaves are not adjacent in each $x_{\mu}$. 

Indeed, if there exist two leaves in one of the $x_\mu$ which are separated by at least 4 internal vertices, their labels correspond to non-adjacent leaves for each $\mu$. 

If any pair of leaves in each $x_\mu$ are separated by fewer than 4 internal vertices, it is sufficient to find in some $x_{\mu_0}$ two leaves $\alpha$ and $\beta$ which are separated by precisely three internal vertices, say $v_1=v(\alpha)$, $v_2$ and $v_3=v(\beta)$ so that at least one of the $v_i$ has valency 4. In this case, the labels $\alpha$ and $\beta$ are not adjacent in any of the $x_\mu$. Such $x_{\mu_0}$ can always be found. Indeed, suppose that in $x_{\mu_0}$ any pair of leaves which are separated by precisely three internal vertices, are separated by \emph{trivalent} vertices. Then, $n=6$ and the only possibility for $x_{\mu_0}$ is the graph b) on the last figure.
This, however, leads to a contradiction since in this case some other $x_{\mu_1}$ would either have a path of 4 internal vertices or a vertex of valency 4.

\section{The space of stable rational curves $\overline{\M}_{0,n}$}

The set $T_{n-1}$ of trees with $n$ marked leaves can be thought of as the combinatorial version of the Deligne-Mumford compactification $\overline{\M}_{0,n}$ of the moduli spaces of rational curves with $n$ marked points (Deligne, Mumford 1969). 

Recall that the moduli space ${\M}_{0,n}$ is the space of all configurations of $n$ distinct points on a complex projective line, considered modulo the action of the group of M\"obius transformations. It has a compactification $\overline{\M}_{0,n}$ which consists of all \emph{stable} rational curves with $n$ marked points. Such a curve  is a tree of projective lines with nodal singularities and $n$ marked points, which has no automorphisms.The marked points are assumed to be distinct from the nodes and among themselves and carry $n$ distinct labels; we may take these labels to be numbers from 0 to $n-1$. The absence of automorphisms means that each line contains at least three distinguished points; that is, either marked points or singularities. The complement to ${\M}_{0,n}$ in $\overline{\M}_{0,n}$ consists of curves with more than one irreducible component. 

For curves with fewer than 5 marked points, the moduli spaces of stable curves are very simple. When $n<4$ one defines $\overline{\M}_{0,n}$ to be a point. Assigning to a quadruple $(z_1, z_2, z_3, z_4)$ of distinct points on $\PP^1=\mathbb{C}\cup\{\infty\}$ its cross-ratio
\begin{equation}\label{cross} \frac{(z_4-z_1)(z_2-z_3)}{(z_4-z_3)(z_2-z_1)}\end{equation}
 we obtain the embedding of ${\M}_{0,4}$ into $\PP^1$ which extends to an isomorphism between $\overline{\M}_{0,4}$ and $\PP^1$.

The first non-trivial case $n=5$ is already quite interesting. In particular, the real part of $\overline{\M}_{0,5}$ is a non-orientable surface with a natural decomposition into 12 pentagons;  this led S.~Devadoss (1999) to characterize it as  ``the evil twin of the dodecahedron'' (in fact, it is a connected sum of 5 projective planes). The cohomology of $\overline{\M}_{0,n}$ for all $n$ has been computed by Keel (1992);  Etingof, Henriques, Kamnitzer and Rains (2010) described the cohomology of the real part. One can write down explicit equations for all the $\overline{\M}_{0,n}$, see the paper by Keel and Tevelev (2009). As we shall see here, one may think of the equations for arbitrary $\overline{\M}_{0,n}$ as ``simplicial consequences'' of the equations for $\overline{\M}_{0,5}$.

\medskip

For each label $i$ there is a forgetful morphism 
\[\partial_i: \overline{\M}_{0,n}\to \overline{\M}_{0,n-1}
\] 
which consists in:
\begin{enumerate}
\item erasing the point marked by $i$ and, for each $j>i$, replacing the label $j$ by $j-1$; 
\item collapsing the component with only two distinguished points if such a component appears after the previous step.
\end{enumerate}

The forgetful morphisms satisfy the simplicial identities:
\[
\partial_i \circ \partial_j = \partial_{j-1} \circ \partial_i
\]
for all pairs of labels $i<j$, and, therefore, the spaces $\overline{\M}_{0,n}$ form a $\Delta$-set (with the space $\overline{\M}_{0,n}$ being the set of $(n-1)$-simplices). 

\begin{thm}\label{modulo}
The sets $\overline{\M}_{0,n}$,  together with the maps $\d_i$ form a $\Delta$-set which is uniquely fillable in dimensions 5 and greater.
\end{thm}

This, in particular, means that the simplicial identities can be thought of the equations for $\overline{\M}_{0,n}$ in a product of $n$ copies of $\overline{\M}_{0,n-1}$ for $n>5$.

\begin{proof}

The space $\overline{\M}_{0,n}$ can be subdivided into strata indexed by the elements of $T_{n-1}$; see (Kock, Vainsencher 2007). Namely, a point in $\overline{\M}_{0,n}$ is uniquely specified by a tree in $T_{n-1}$ each of whose $k$-valent internal vertices is labelled by a configuration in $\M_{0,k}$; the labels of the points of each configuration are the edges emanating from the corresponding vertex. Note that, since $\M_{0,k}$ is a one-point space for $k<4$, the difference between $\overline{\M}_{0,n}$ and $T_{n-1}$ consists in the labels at the vertices of valency 4 and more.

The effect of the map $\partial_\alpha$ on $\overline{\M}_{0,n}$ amounts to that of $\partial_\alpha$ on $T_{n-1}$ together with forgetting the corresponding point in $\M_{0,k}$ for the vertex $v(\alpha)$ when $v(\alpha)$ is at least 4-valent. The question whether a point $x\in \overline{\M}_{0,n}$ can be uniquely reconstructed from $\partial_\alpha(x)$ and $\partial_\beta(x)$ has a somewhat simpler answer than in $T_{n-1}$: this can be always be done uniquely unless $v(\alpha)$ and $v(\beta)$ are both trivalent and connected by an edge. Other than this, no changes are necessary in the proof of Theorem~\ref{combin} in order to adapt it for $\overline{\M}_{0,n}$.
\end{proof}

In fact, the embedding of $\overline{\M}_{0,5}$ into $(\PP^1)^5$ defined as the product $\d_0\times\ldots\times\d_4$ is also injective, although its image is not given by the simplicial identities alone (which are trivial in this case). The following is well-known:

\begin{prop}\label{mfive}
The image of $\overline{\M}_{0,5}$ in $(\PP^1)^5$ is the non-singular surface given by the equations
\begin{eqnarray*}
a_1(a_4 b_5-a_5 b_4)&=&b_1 b_5 (a_4 - b_4)\\
a_2(a_4 b_5-a_5 b_4)&=&b_2 a_4 b_5\\
a_3(a_4 b_5-a_5 b_4)&=&b_3 a_4 (b_5- a_5),
\end{eqnarray*}
where $[a_k: b_k]$, for $1 \leq k\leq 5$, are the homogeneous coordinates in the $k$th copy of $\PP^1$.
\end{prop}

\begin{proof} 
For a point on ${\M}_{0,5}$, that is, an ordered quintuple $z=(z_1,z_2,z_3,z_4,z_5)$ of distinct points on $\PP^1$, we have that $a_i/b_i\in \mathbb{C}\cup\{\infty\}$ is the cross-ratio of the quadruple obtained by omitting $z_i$ from $z$; verifying the above equations is a straightforward matter. Since ${\M}_{0,5}$ is open in $\overline{\M}_{0,5}$, these equations are also satisfied on the image of $\overline{\M}_{0,5}$.

Conversely, if a point  $c=(c_1,c_2,c_3,c_4,c_5)$ of $(\PP^1)^5$ satisfies these equations, the corresponding  curve $x\in\overline{\M}_{0,5}$ can be reconstructed as follows. The number of projective lines of $x$ is: one if none of the $c_i$ is 0, 1 or $\infty$, two if exactly three of the $c_i$ are 0, 1 or $\infty$, and three if all of the $c_i$ are equal to 0, 1 or $\infty$. The entries equal to $0,1$ or $\infty$ determine the combinatorics of the marked tree and the $c_i$ different from 0,1 and $\infty$ gives in each case the cross-ratios of the marked points in each projective line.

The image of $\overline{\M}_{0,5}$ can be covered by explicit non-singular charts obtained by fixing three of the five points on $\PP^1$ to be $0,1,\infty$. For instance, ordered quintuples of the form $(0,1,\infty, x, y)$  with $x,y\in\mathbb{C}$ define the chart
$$(x,y) \mapsto \left(\frac{y-1}{y-x}, \frac{y}{y-x}, \frac{y(1-x)}{y-x}, y, x\right);$$
the other charts differ by the indices of the fixed points. 
\end{proof}

The equations for $\overline{\M}_{0,5}$, together with the simplicial identities, produce the equations for all the $\overline{\M}_{0,n}$. For instance, consider the case $n=6$. The moduli space $\overline{\M}_{0,6}$ is a subvariety of 
$$(\overline{\M}_{0,5})^6\subset ((\PP^1)^{5})^6.$$
Denote the $[a_{ij} : b_{ij}]$, where 
$1\leq i\leq 5$ and $1\leq j\leq 6$, the homogeneous coordinates in $(\PP^1)^{30}$, with the index $j$ being the number of the copy of $\overline{\M}_{0,5}$ and $i$ the number of the coordinate in the corresponding copy of $(\PP^1)^5$. The simplicial identities give rise to the equalities
$$[a_{ij} : b_{ij}]=[a_{(j-1)i} : b_{(j-1)i}]$$
whenever $i<j$. Therefore, the complete set of equations for $\overline{\M}_{0,6}$ in $(\PP^1)^{30}$ is
\begin{eqnarray*}
a_{1j}(a_{4j} b_{5j}-a_{5j} b_{4j})&=&b_{1j} b_{5j} (a_{4j} - b_{4j})\\
a_{2j}(a_{4j} b_{5j}-a_{5j} b_{4j})&=&b_{2j} a_{4j} b_{5j}\\
a_{3j}(a_{4j} b_{5j}-a_{5j} b_{4j})&=&b_{3j} a_{4j} (b_{5j}- a_{5j})\\
a_{ij}b_{(j-1)i} &= &a_{(j-1)i}b_{ij}
\end{eqnarray*}
where $i,j$ vary over the set $1\leq i< j \leq 6$.

\section{Other examples}

There are other varieties similar to the moduli spaces of stable rational curves whose points can be thought of as trees with marked leaves and ``decorations'' at the internal vertices. The two principal examples are two compactifications of the configuration space $F_n(X)$ of $n$ distinct points on an algebraic variety $X$: namely, the Fulton-MacPherson compactification $X[n]$ (Fulton, Macpherson 1994), and Ulyanov's (2002) polydiagonal compactification $X\langle n\rangle$.


The configuration space $F_n(X)$ is defined as the complement in $X^n$ to the union of all the diagonals $z_i=z_j$. The spaces $F_n(X)$ form a $\Delta$-set: the map $\d_i$ erases the $i$th point in the configuration. It is easy to see that this $\Delta$-set is uniquely fillable in dimensions two and higher.

A point in $X[n]$ is a collection $(z_1,\ldots,z_n)\in X^n$ together with additional data: if two or more of the $z_i$ coincide at a point $z\in X$, one specifies a \emph{screen} at $z$. Denote by $I\subseteq \{1,\ldots,n\}$ the set of indices of the $z_i$ which coincide with $z$. A screen at $z$ is a configuration of points, labelled by the set $I$ and not all equal to each other, in the tangent space $T_z X$; it is considered up to a translation and a multiplication by a nonzero scalar. If, in turn, some of the points in the screen coincide, one specifies another screen which corresponds to the set of coinciding points, and the procedure is iterated until in some screen all the points corresponding to different indices are distinct (Fulton, Macpherson, 1994, page 191).
The map $\partial_i$ extends from $F_n(X)$ to $X[n]$: it erases $z_i$ from $(z_1,\ldots,z_n)$ and deletes the corresponding points from all the screens; if the index $i$ happens to occur in some screen with only two labels, this screen is also erased.
It is clear that the $\d_i$ satisfy the simplicial identities.

\begin{prop}\label{FM}
The spaces $X[n]$ form a $\Delta$-set which is uniquely fillable in dimensions three and greater. 
\end{prop}

This result should not be surprising: for instance, Fulton and MacPherson (1994) explicitly point out that $X[n]$ form a $\Delta$-set (without using this terminology) and that $X[n]$ is a subvariety in a product of several copies of $X[2]$ and $X[3]$.

The proof (whose details we omit) is very similar to the case of $\overline{\M}_{0,n}$. Indeed, a point of $X[n]$ can be represented by a forest of rooted trees with no bivalent vertices. The roots are univalent and marked  by distinct points of $X$; the rest of the leaves are numbered from 1 to $n$; each internal vertex carries a label corresponding to a screen. The points on the screen at any internal vertex are labelled by the outgoing edges, assuming that every edge is oriented away from the root. Again, since a screen with two points in it is unique, it is sufficient to consider the labels only for the internal vertices of valency at least 4. 

\medskip

The points that are added to $F_n(X)$ in the construction of $X[n]$ carry the data that record the directions and the hierarchy of the collisions of several points. The polydiagonal compactification is a generalization of the Fulton-MacPherson compactification that allows to record, in addition, the velocities of collisions among several collisions. A point in $X\langle n\rangle$ is given by a forest of rooted trees as in the construction of $X[n]$, with the following differences: 
\begin{enumerate}
\item
there is a total order on the set of internal vertices which can be expressed by a \emph{level} function which increases in the direction away from the root; 
\item
for each screen, a non-zero real \emph{scale factor} is given;
\item 
the screens, rather than being considered up to up to a translations and dilatations have a finer equivalence on them; namely, one is allowed to
\begin{enumerate}
\item apply a translation to all the points in one screen;
\item apply a dilatation by a non-zero real $\lambda$ 
of all the points in  one screen and, at the same time, multiply its scale factor by $\lambda^{-1}$;
\item multiply the scale factor of all the screens on the same level by a non-zero number.
\end{enumerate}
\end{enumerate}
Then, again, we have the forgetful maps $\d_i: X\langle n\rangle\to X\langle n-1\rangle$ which satisfy the simplicial identities.

\begin{prop}\label{U}
The $X\langle n\rangle$ form a $\Delta$-set, uniquely fillable in dimensions four and greater.
\end{prop}

Here, the unique fillability dimension four, as opposed to three in the Fulton-MacPherson case, is due to the presence of the scale factors. For instance, consider two points of $X\langle 4\rangle$ which correspond to
the forest 
\[\includegraphics[width=100pt]{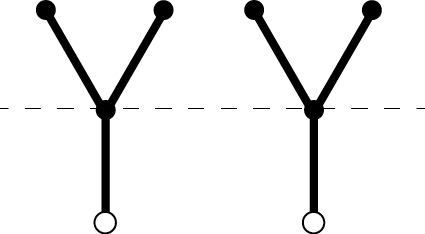}\] 
with the same markings of roots and leaves but with different (even after any rescaling) scale factors.
These points will map to the same elements in $X\langle 3\rangle$ under each $\d_i$, since erasing any leaf destroys the scale factors. It can be seen that this problem does not arise when $n>4$.

\subsection*{Acknowledgements}
This note grew out of the MSc thesis of the first named author. We would like to thank Vladmir Dotsenko for comments. 



\end{document}